\documentclass[a4paper, 12pt]{amsart}
\newcommand{\newdefinition}{\theoremstyle{definition}\newtheorem}
\newcommand{\newproof}[2]{}

\usepackage[utf8x]{inputenc}
\usepackage{amsmath,amssymb,microtype}
\usepackage{pifont,MnSymbol}
\usepackage{tikz}
\usepackage[colorlinks=true]{hyperref} 
\usepackage{tikz-cd}
\usepackage[toc,title]{appendix}
\usepackage{amscd}

\usetikzlibrary{matrix,arrows,decorations}

\DeclareMathOperator{\Set}{Set}
\DeclareMathOperator{\Prim}{P}
\DeclareMathOperator{\Ringk}{Ring_k}
\DeclareMathOperator{\Alg}{Alg_k}

\DeclareMathOperator{\Schk}{Sch_k}
\DeclareMathOperator{\Bialgk}{Bialg_k}
\DeclareMathOperator{\Algk}{Alg_k}
\DeclareMathOperator{\Algkend}{Alg_k^{end}}
\DeclareMathOperator{\Modk}{Mod_k}
\DeclareMathOperator{\Bialgkp}{Bialg_k^p}
\DeclareMathOperator{\Sym}{S}
\DeclareMathOperator{\End}{End}
\DeclareMathOperator{\Aut}{Aut}
\DeclareMathOperator{\BR}{BR_{k,k}}

\newcommand{\ZZ}{\mathbb{Z}}

\newcommand{\Ring}{{\operatorname{Ring}}}
\newcommand{\Hom}{{\operatorname{Hom}}}
\numberwithin{equation}{section}

\newtheorem{lemma}[equation]{Lemma}

\newtheorem{corollary} [equation]{Corollary}

\newdefinition{example}[equation]{Example}
\newdefinition{defn}[equation]{Definition} 
\newdefinition{remark}[equation]{Remark}
\newdefinition{notation}[equation]{Notation}
\newproof{proof}{Proof}

\newtheorem{theorem}{Theorem}[section]

\theoremstyle{definition}

\newtheorem{definition}[theorem]{Definition}

\makeatletter

\def\@setcopyright{}

\def\serieslogo@{}

\makeatother

\begin{document}


%






\author{Magnus Carlson}

\address{}

\email{macarlso@math.kth.se}



\newcommand{\Spec}{\text{Spec }}

\title[Classification of plethories in characteristic zero]{Classification of plethories in characteristic zero}





\begin{abstract}
\noindent We classify plethories over fields of characteristic zero, thus answering a
question of Borger-Wieland and Bergman. All plethories over
characteristic zero fields are linear, in the sense that they are free
plethories on a bialgebra. For the proof we need some facts from the theory of ring schemes where we extend previously known results. We also classify plethories with trivial Verschiebung over a
perfect field of non-zero characteristic and indicate future work. 
\end{abstract}


\date{\today}



\maketitle
\tableofcontents




\section{Introduction}
\noindent Plethories, first introduced by Tall-Wraith \cite{TallWraith}, and then studied by Borger-Wieland \cite{BW}, are precisely the objects which act on $k$-algebras, for $k$ a commutative ring. There are many fundamental questions regarding plethories which remain unanswered. One such question is, given a ring $k,$ whether one can classify plethories over $k,$  in this paper we will take a first step towards a classification.  \\

\noindent For some motivation, let us start by looking at the category of modules $\Modk$ over a commutative ring $k.$  If we consider the category of representable functors $\Modk \rightarrow \Modk,$ there is a monoidal structure given by composition of functors. Then one defines a $k$-algebra $R$ as a $k$-module $R$ together with a comonad structure on the representable endofunctor $\Modk(R,-):\Modk \rightarrow \Modk$ with respect to composition of functors. Heuristically, this says that a $k$-algebra is precisely the kind of object which knows how to act on $k$-modules. This can be extended to a non-linear setting, so that instead of looking at $k$-modules we look at $k$-algebras $\Algk$ and consider representable endofunctors $\Algk \rightarrow \Algk.$ A comonoid with respect to composition of functors is then called a plethory and analogously, a plethory is what knows how to act on $k$-algebras. One particular important example of a plethory is the $\ZZ$-algebra $\Lambda$ which consist of the ring of symmetric functions in infinitely many variables with a certain biring structure. The functor $\Algk(\Lambda,-):\Algk \rightarrow \Algk$  represents the functor taking a ring $R$ to its ring of Witt vectors. Using plethories one gets a very conceptual view of Witt vectors and in \cite{BorgerGeom} James Borger develops the geometry of Witt vectors using the plethystic perspective. \\ 

\noindent Let now $k$ be a field. If we let $\mathcal{Pl}_k$ denote the category of plethories over $k,$  there is a forgetful functor $$F:\mathcal{Pl}_k \rightarrow \Bialgk$$ into the category of cocommutative counital bialgebras over $k.$ This functor has a left adjoint $\Sym(-):\Bialgk \rightarrow \mathcal{Pl}$ and we say that a plethory $P$ is linear if $P \cong \Sym(Q)$ for some cocommutative, counital bialgebra $Q.$ Heuristically, a plethory $P$ is linear if every action of $P$ on an algebra $A$ comes from an action of a bialgebra on $A.$ The main theorem of this paper is: 
\begin{theorem}
Let $k$ be a field of characteristic zero. Then any $k$-plethory is linear.
\end{theorem} 
\noindent This answers a question of Bergman-Hausknecht \cite[p.336]{BergmanHaus} and Borger-Wieland \cite{BW} in the positive. The theorem is proved by studying the category of affine ring schemes. We have the following results, extending those of Greenberg \cite{Greenberg63} to arbitrary fields and not necesarily reduced schemes:
\begin{theorem}
Let $k$ be a field. Then any connected ring scheme of finite type is unipotent.
\end{theorem}
\begin{theorem}
Let $P$ be a connected ring scheme of finite type over $k.$ Then $P$ is affine. 
\end{theorem}
\noindent For the case of characteristic $p>0$ our classification results on plethories are not as complete and further work is needed to have a complete classification. To explain our classification results here we need some definitions. Let $F_k$ be the Frobenius homomorphism of $k$ and $k\langle F\rangle$ be the non-commutative ring which as underlying set is $k[F]$ and has multiplication given by $$F^i F^j = F^{i+j}$$ and $$Fa = F_k(a)F.$$ We define $\Bialgkp$ to be the category of cocommutative, counital bialgebras over $k$ which also are modules over $k \langle F \rangle.$ Once again, for a plethory over a perfect field $k$ of char $k>0$ there is a forgetful functor $$\mathcal{Pl}_k \rightarrow \Bialgkp$$ which has a left adjoint $\Sym^{[p]}.$ Call a plethory $P$ $p$-linear if $P \cong \Sym^{[p]}(Q)$ for some $Q \in \Bialgkp.$ We have then the following classification result:
\begin{theorem}
Let $k$ be a perfect field of characteristic $p>0.$ Assume that $P$ is a plethory over $k$ such that the Verschiebung $V_P=0.$ Then $P$ is $p$-linear.  
\end{theorem}
\noindent The structure of this paper is as follows. In section $2$ we study ring schemes and prove some results which we will need for our classification theorem. The main theorems of this section that are needed for later purposes are Theorem \ref{ringunipotent} and Theorem \ref{ringfiltered}. In section $3$ we introduce plethories and $k-k$-birings and provide some examples. This section contains no new results and gives just a brief introduction to the relevant objects as defined in Borger-Wieland \cite{BW}. In section $4$ we prove that all plethories over a field $k$ of characteristic zero is linear using the results from section $2.$ We also show that any $k-k$-biring is connected. In section $5$ we prove some initial classification results regarding plethories in characteristic $p>0.$	 
\section*{Notation and conventions}
\begin{list}{}{\labelwidth 8ex \leftmargin 10ex \labelsep 2ex \itemsep 2pt \parsep 0pt}
\item[{$\Ring$}] category of commutative and unital rings. 
\item[{$\BR$}] category of $k-k$-birings.
\item[{$\mathcal{Pl}_k$}] category of $k$-plethories. 
\item[{$\Bialgk$}] category of cocommutative $k$-bialgebras.
\item[{$\Bialgkp$}] category of cocommutative $k$-$p$-bialgebras.
\item[{$\odot$}] composition product of $k-k$-birings, Def.  3.2.
\item[{$\Algk$}] category of commutative algebras over the ring $k.$ 
\item[{$\Delta^+_A,\Delta^\times_A$}] coaddition resp. comultiplication map for a biring $A.$ 
\item[{$\epsilon^+_A,\epsilon^\times_A$}] counit for coaddition resp. comultiplication for a biring $A.$ 
\item[{$\beta_A$}] co-$k$-algebra strucutre on a $k-k$-biring $A.$ 
\item[{$\Delta_2^+, \Delta_2^\times$}] abbreviation for the composite $(1 \otimes \Delta^+)\circ \Delta^+$ resp. $(1 \otimes \Delta^\times) \circ \Delta^\times.$
\item[{$P$}] primitive elements functor 
\item[{$\mathcal{O}_X$}] structure sheaf of a scheme $X.$ 
\item[{$\Schk$}] category of $k$-schemes for $k$ a commutative ring. 
\item[{$\mathbb{G}_a$}] the affine line viewed as a group scheme, see Ex. 3.1
\item[{$\mathbb{G}_m$}] the multiplicative group scheme, after Def. 2.5.
\item[{$\mu_p$}] the p-th root of unity group scheme, Ex. 5.2
\item[{$\alpha_p$}] see Ex. 3.3
\item[{$\pi_0(G)$}] the group scheme of connected components of a group scheme $G$ over the field $k,$  Def. 4.2
\item[{$\Sym$}] free plethory functor on a cocommutative bialgebra. Def. 4.1
\item[{$\Sym^{[p]}$}] free plethory functor on a cocommutative $p$-bialgebra, after Def. 5.1.
\item[{$G^\circ$}] the identity component of a group scheme $G.$ 
\item[{$F_G,V_G$}] the Frobenius resp. Verschiebung morphism of a group scheme $G$ over a perfect field of characteristic $p>0.$ 
\item[{$k\langle F \rangle$}] the twisted polynomial algebra.
\end{list}
For us, all rings are commutative and unital. We will use Sweedler notation for coaddition $\Delta^+$ and $\Delta^\times$, so that $\Delta^+(x) = \sum_i  x_i^{(1)} \otimes x_i^{(2)}$ and $\Delta^\times(x) = \sum_i x_i^{[1]} \otimes x_i^{[2]}$  if $x \in A$ where $A$ is a biring. For concepts from the theory of group schemes not introduced properly here, we refer to \cite{MilneiAG} or \cite{DemazureG} .
\subsection*{Acknowledgements}
I am very grateful to James Borger who supplied me with the conjecture for characteristic zero and the idea of ''weakly linear'' plethories. He has been more than generous with his knowledge and many of the ideas in this paper come from conversations with him. I also thank David Rydh and Lars Hesselholt for the many useful comments they gave which helped improve the quality of this article. I would also like to thank my advisor Tilman Bauer for his support and for his many thoughtful suggestions on this article.
\section{Ring schemes}
\noindent Let $k$ be a commutative ring. Recall that $\mathcal{R}$ is a ring scheme over $k$
if $\mathcal{R}$ is a scheme together with a lift of the functor $$\Schk(-,\mathcal{R}):\Schk \rightarrow \Set$$
to $\Ring.$ We say
that a ring scheme is a $k$-algebra scheme if the lift factors through the category of $k$-algebras. We will mostly be concerned with affine
ring schemes. Ring schemes were studied by Greenberg in \cite{Greenberg63} and he showed that for connected, reduced ring schemes of finite
type over an algebraically closed field $k,$ the underlying scheme
is always affine. Further, he shows that the underlying group variety is
always unipotent. We improve on these results by showing that any connected ring schemes of finite type over
an arbitrary field is affine, and that the 
underlying group scheme is always unipotent. From now on, in this section, $k$ is always
a field.
\begin{definition}
A $k$-scheme $X$ is anti-affine if $\mathcal{O}_X(X)=k.$ We say that a group
scheme is anti-affine if its underlying scheme is anti-affine
\end{definition}
\noindent
For example, abelian varieties are all anti-affine group schemes. An anti-affine group scheme has the property that any morphism from it into an affine group scheme is trivial. Anti-affine groups are very important for the structure of group schemes as the following theorem shows:
\begin{theorem} [{\cite[Theorem 1]{BrionStr}}] \label{Brionstructure}
If $G$ is a group scheme of finite type over a field $k$ there is an exact sequence of group schemes $$0 \rightarrow G^{ant} \rightarrow G \rightarrow G/G^{ant} \rightarrow 0$$ such that $G^{ant}$ is anti-affine and $G/G^{ant}$ is affine.
\end{theorem}
\noindent We will now want to show that all connected finite type ring schemes are affine, i.e that in the above exact sequence $G^{ant} = \Spec k.$ For this, we will need the following lemma.
\begin{lemma}
Let $X,Y,Z$ be $k$-schemes with $X$ quasi-compact and anti-affine and
$Y$ locally noetherian and irreducible.Suppose that $f:X \times Y
\rightarrow Z$ is a morphism such that there exist k-rational points
$x_0 \in X,$ $y_0 \in Y$ such that $f(x,y_0) =f(x_0,y_0)$ for all $x.$
Then $f(x,y)=f(x_0,y)$ for all $x,y.$
\end{lemma}

\begin{proof}
see \cite[Lemma 3.3.3.]{BrionStr}. \end{proof}
\begin{theorem} \label{ringaffine}
Let $\mathcal{R}$ be a connected ring scheme of finite type over $k.$ Then $\mathcal{R}$ is affine.
\end{theorem}
\begin{proof}
We know by Theorem \ref{Brionstructure} that $\mathcal{R}$ sits in the middle of an extension of an affine group scheme by an anti-affine group. Let
$$0 \rightarrow \mathcal{R}^{ant} \rightarrow \mathcal{R} \rightarrow \mathcal{R}^{aff} \rightarrow 0$$
be the corresponding extension where $\mathcal{R}^{ant}$ is anti-affine and $\mathcal{R}^{aff}$ the affine quotient. Since $\mathcal{R}$ is of finite type, $\mathcal{R}^{aff}$ is a ring scheme. This follows from the fact that if $X,Y$ are quasi-compact $k$-schemes, then the obvious map $\mathcal{O}_X(X) \otimes_k \mathcal{O}_Y(Y) \rightarrow \mathcal{O}_{X \times_k Y}(X \times_k Y)$ is an isomorphism (see \cite[Lemma 2.3.3.]{BrionStr}). This implies that $\mathcal{R}^{ant}$ defines an
ideal scheme in $\mathcal{R},$ i.e for all rings $S$ over $k, \mathcal{R}^{ant}(S)$ is an
ideal of $\mathcal{R}(S).$ Now, we will apply the above lemma with $Y=\mathcal{R}$ (note that $\mathcal{R}$ is irreducible) and
$X=Z=\mathcal{R}^{ant}.$ Taking $x_0=e_{\mathcal{R}^{ant}}$ and $y_0=e_{\mathcal{R}} $ to be the rational points corresponding to the additive identity of $\mathcal{R}^{ant}(k)$ and $\mathcal{R}(k)$ respectively, we have that
$m(x,y_0)=m(x_0,y_0)$ is identically equal to zero. Thus, we have that
$m(x,y)=m(x_0,y)$ is identically zero. But, letting $1_{\mathcal{R}}$ be the
rational point corresponding to the multiplicative identity of $\mathcal{R}(k)$ we
have that $m(x,1_{\mathcal{R}})$ is zero. But multiplication by $1$ is always
injective, and thus, $\mathcal{R}^{ant}$ is trivial and $\mathcal{R}$ is affine. \end{proof}
\noindent We don't know if the condition for $\mathcal{R}$ to be of finite type is necessary in \ref{ringaffine}. Let us recall the following definition from the theory of algebraic groups.
\begin{definition}
Let $G$ be a commutative group scheme over $k$. We say that $G$ is unipotent if it is affine and if every non-zero closed subgroup $H$ of $G$ admits a non-zero homomorphism $H \rightarrow \mathbb{G}_a.$
\end{definition}
\noindent The data of a homomorphism $G \rightarrow \mathbb{G}_a$ is the same as specifying an element $x \in A_G$ in the underlying Hopf algebra of $G$ that satisfies $\Delta_G(x) = x \otimes 1 + 1 \otimes x,$ i.e specifying a primitive element. If $G = \Spec A_G$ is an affine group scheme and $A_G$ the Hopf algebra associated to $G,$ then saying that $G$ is unipotent is the same as saying that it is coconnected (or conilpotent). The following definition will be useful for the proof of Theorem \ref{ringunipotent}. 
\begin{definition}
Let $G$ be a commutative affine group scheme over a field. We say that $G$ is multiplicative if every homomorphism $G \rightarrow \mathbb{G}_a$ is zero.
\end{definition}
\noindent An example of a multiplicative group is $\mathbb{G}_m = \Spec k[x,x^{-1}].$ There can in general be no homomorphism from a multiplicative group into a unipotent group and no morphisms from a unipotent group to a multiplicative group (for a proof, see \cite[Corollary 15.19-15.20]{MilneiAG}).\\ 
\noindent The following theorem was shown for reduced ring varieties over an algebraically closed fields by
Greenberg, but the results carry over for perfect fields without any modification. We improve on this by carrying through the proof when $\mathcal{R}$  is not necesarily reduced and over any field $k$. Further,the theorem can be extended to ring schemes not necesarily of finite type if
the ring scheme is already known to be affine. 
\begin{theorem} \label{ringunipotent}
Over a field $k,$ all connected ring schemes $\mathcal{R}$ of
finite type are unipotent.
\end{theorem}
\begin{proof}
By Theorem \ref{ringaffine}, a connected ring scheme is affine. We know that $\mathcal{R}$ contains a greatest multiplicative subgroup $\mathcal{R}_m$ that has the property that for all endomorphisms $\alpha$ of $\mathcal{R}_S,$ (where $\mathcal{R}_S$ is the base change of $\mathcal{R}$ to $S$) for $S$ a $k$-algebra, that $\alpha((\mathcal{R}_m)_S) \subset (\mathcal{R}_m)_S$ (\cite[Theorem 17.16]{MilneiAG}). Thus, since any $x \in \mathcal{R}(k)$ defines an endomorphism of $\mathcal{R}$ (as a group scheme) through multiplication by $x,$ we have that $\mathcal{R}_m$ is an ideal of $\mathcal{R}.$ It is known that any action of a connected algebraic group on a multiplicative group must be trivial, i.e for $G$ connected and $H$ multiplicative, a map $G \rightarrow \Aut(H,H)$ must have image the identity. We will need the following, which says that any map $G \rightarrow \End(H,H)$ where $G$ is any connected group scheme and $H$ is multiplicative is trivial. This is basically just deduced, mutatis mutandis, from the proof of \cite[Theorem 14.28]{MilneiAG}.  So, we see that $0$ and $1$ defines the same endomorphisms on the ideal scheme $\mathcal{R}_m.$ But this is only the case if $\mathcal{R}_m=0.$  \end{proof}
\noindent To extend this to all connected ring schemes, we need the following:
\begin{theorem} \label{ringfiltered}
Let $k$ be a field and $\mathcal{R}$ be an affine ring scheme over $k.$ Then $\mathcal{R}$
is a filtered limit of ring schemes of finite type.
\end{theorem}
\begin{proof}
The following proof is inspired by the analogue theorem for Hopf algebras over a field, as occurs in for example Milne \cite[Prop. 11.32]{MilneiAG}. Write $\mathcal{R} = \Spec A_\mathcal{R}.$ We know that $A_\mathcal{R}$ is a bialgebra and we see that
we can reduce to proving that any $a \in A_\mathcal{R}$ is contained in a
sub-bialgebra of finite type. Let $\Delta^+:A_\mathcal{R} \rightarrow A_\mathcal{R} \otimes
A_\mathcal{R}$ be the coaddition giving the additive group structure on $\mathcal{R}$ and
$\Delta^\times: A_\mathcal{R} \rightarrow A_\mathcal{R} \otimes A_\mathcal{R}$ the comultiplication
defining the multiplication on $\mathcal{R}.$ Consider $$ \Delta^+_2(a)= \sum_{i,j}
c_i \otimes x_{ij} \otimes d_j $$ with $c_i$ and $d_j$ linearly
independent. Now, by the fundamental theorem of coalgebras, we know that
if we take $X$ to be the subspace of $A_\mathcal{R}$ generated by $\{x_{ij}\},$
then this is a subcoalgebra, i.e that $\Delta^+(x_{ij}) \subset X
\otimes X.$ Now, for each $x_{ij}$ in this system, consider
$$\Delta^{\times}_2(x_{ij}) = \sum_{k,l} e_i \otimes y_{kl} \otimes f_l$$
with $e_i$ and $f_l$ linearly independent. With the same arguments, one
sees that for the subspace $Y$ generated by $\{y_{kl}\}$ we have
$\Delta^{\times}(y_{kl}) \subset Y \otimes Y.$ Let now $Z$ be subalgebra generated by the finite-dimensional subspace spanned by
 $\{x_{ij}, y_{kl}\}.$ We claim 
that $Z$ actually is closed under both the operation $\Delta^+$ and
$\Delta^{\times}.$ It is clear that $$\Delta^{\times}(x_{ij}) \subset Z
\otimes Z$$ and the same holds for coaddition. It is also easy to verify that $\Delta^\times(y_{kl}) \subset Z \otimes Z.$ We will now prove that
$\Delta^{+}(y_{kl}) \subset Z \otimes Z$ and for this, consider the
following diagram which is easily verified if we reverse all arrows and think of it in terms of rings.  \\
\begin{tikzcd}A_\mathcal{R} \arrow[r, "\Delta^\times"] \arrow[d, "\Delta^+"]
& A_\mathcal{R} \otimes A_\mathcal{R} \arrow[d, equals] \\
A_\mathcal{R} \otimes A_\mathcal{R} \arrow[d,"\Delta^\times \otimes \Delta^\times"]
& A_\mathcal{R} \otimes A_\mathcal{R}  \arrow[d,equals]\\
A_\mathcal{R} \otimes A_\mathcal{R} \otimes A_\mathcal{R} \otimes A_\mathcal{R} \arrow[d,"1 \otimes T \otimes 1"] & 
A_\mathcal{R} \otimes A_\mathcal{R} \arrow[d,equals] \\
A_\mathcal{R} \otimes A_\mathcal{R} \otimes A_\mathcal{R} \otimes A_\mathcal{R} \arrow[d,"\Delta^\times \otimes 1 \otimes 1 \otimes \Delta^\times"] & 
A_\mathcal{R} \otimes A_\mathcal{R} \arrow[d,equals] \\ 
A_\mathcal{R} \otimes (A_\mathcal{R} \otimes A_\mathcal{R}) \otimes (A_\mathcal{R} \otimes A_\mathcal{R}) \otimes A _\mathcal{R} \arrow[d,"1 \otimes T \otimes T \otimes 1"] & A_\mathcal{R} \otimes A_\mathcal{R}  \arrow[d,"\Delta^\times \otimes 1"] \\
(A_\mathcal{R} \otimes A_\mathcal{R}) \otimes A_\mathcal{R} \otimes A_\mathcal{R} \otimes (A_\mathcal{R} \otimes A_\mathcal{R}) \arrow[d, "M \otimes 1 \otimes 1 \otimes M"] & 
(A_\mathcal{R} \otimes A_\mathcal{R}) \otimes A_\mathcal{R} \arrow[d, "1 \otimes \Delta^+ \otimes 1"] \\
A_\mathcal{R} \otimes A_\mathcal{R} \otimes A_\mathcal{R} \otimes A_\mathcal{R} \arrow[r,equals] & A_\mathcal{R} \otimes A_\mathcal{R} \otimes A_\mathcal{R} \otimes A 
\end{tikzcd} 
\\ \\ \noindent
Here $M$ is the multiplication map and $T$ swithces the factors. What the diagram is saying, is just relating different ways of forming $$abd+acd$$ for $a,b,c,d$ in a ring. So this says, that  $$(1 \otimes \Delta^+ \otimes 1)(\Delta^{\times}_2(x_{ij}) = \sum_{k,l} e_k \otimes \Delta^+(y_{kl}) \otimes f_l  \in Z \otimes Z \otimes Z \otimes Z.$$
Now, since $e_k$ are independent, this means that $$\sum_l
\Delta^+(y_{kl}) \otimes f_l \in Z \otimes Z \otimes Z$$ and by linear independence
of each $f_l$ this means that $$\Delta^+(y_{kl}) \in Z \otimes Z.$$ Now, let $W$ be
the sub-algebra generated by $Z \cup S(Z)$ where $S:A_\mathcal{R} \rightarrow A_\mathcal{R}$ is
the antipode. It is easily verified that $$\Delta^+ \circ S = (S \otimes
S ) \circ \Delta^+$$ and that $$\Delta^{\times}(S(Z)) \subset W$$ 
follows from the identity $$\Delta^{\times} \circ S = (1
\otimes S) \circ \Delta^{\times}.$$ We thus see that $W$ is a bialgebra
and we are done. \end{proof}
\begin{corollary}
Any affine connected ring scheme over a field is unipotent.
\end{corollary}
\begin{proof}
Indeed, we know that we can write $P = \varprojlim P_i $ where $P_i$ ranges over ring schemes of finite type. Now, unipotence is stable under inverse limits and this immediately gives that $P$ is unipotent. 
\end{proof}
\section{Plethories and $k-k$-birings.}

\noindent Let $k$ be an arbitrary commutative ring. In this section we will recall
the definition of a plethory as defined in \cite{BW}.

\begin{definition}

A $k$-biring $A$ is a coring object in the category of $k$-algebras.
Explicitly, $A$ is a $k$-algebra together with maps $$\Delta^+: A
\rightarrow A \otimes_k A,$$ $$\Delta^\times : A \rightarrow A \otimes_k
A,$$ $$S:A \rightarrow A,$$ $$\epsilon^+:A \rightarrow k$$and
$\epsilon^\times:A \rightarrow k$ such that: \\ \begin{itemize}

\item The triple $(\Delta^+,\epsilon^+,S)$ defines a cocommutative Hopf
algebra structure on $A$ with $S$ the antipode and $\epsilon^+$ the counit.

\item $\Delta^\times$ is cocommutative coassociative and codistributes
over $\Delta^+$ and $\epsilon^\times:A \rightarrow k$ is a counit for
$\Delta^\times.$

\end{itemize} 
\noindent We say that $A$ is a $k-k$-biring if, in addition to the above data, it
has a map $$\beta:k \rightarrow \Ringk(A,k)$$ of rings, where we endow
$\Ringk(A,k)$ with the ring structure induced from the coring structure
on $A.$
\end{definition}

\noindent Equivalently, a $k-k$-biring $A$ is just an affine scheme together with a lift of
the functor $\Ringk(A,-)$ to the category of $k$-algebras, i.e it is an affine
$k$-algebra scheme. 
\begin{example}
\noindent Let define the $k$-algebra scheme which we will call $\mathbb{G}_a.$ $\mathbb{G}_a$ will represent the identity functor $\Ringk \rightarrow \Ringk.$ The underlying scheme of $\mathbb{G}_a$ is  $\Spec k[e].$  The coaddition and comultiplication is given by $\Delta^+(e) = e
\otimes 1 + e \otimes 1,$ $\Delta^\times(e) = e \otimes e,$ the additive resp. multiplicative counit by 
$\epsilon^+(e)=0,\epsilon^\times(e)=1$ the antipode by $S(e)=-e$ and the co-$k$-linear structure by $\beta(c)(e) = c$ for all
$c \in k.$
\end{example}

\begin{example}
Consider $\mathbb{Z}[e,x].$ On $e,$ we define all the structure
maps as in the previous example. We then define $$\Delta^+(x) =
x \otimes 1+ 1 \otimes x,$$ $$\Delta^\times(x) =
x \otimes e +e \otimes x$$ and
$\epsilon^\times(x)=\epsilon^+(x)=0,$ $S(x) = -
x.$ This $\mathbb{Z}$-ring scheme represents the functor taking a
ring $R$ to $R[\epsilon]/(\epsilon^2)$ , the ring of dual numbers
over that ring.
\end{example}

\begin{example}
Let $k = \mathbb{F}_q$ be a finite field of characteristic $p$ and
consider $$\alpha_p = \Spec k[e]/(e^p)$$ as a group scheme where the group
structure is induced from $\Spec k[e].$ Define a multiplication $$\alpha_p
\times \alpha_p \rightarrow \alpha_p$$ by saying that $xy=0$ for any $x,y \in
\alpha_p(R)$ for $R$ a $k$-algebra. Consider now the constant group scheme
$$\underline{\mathbb{F}_q}= \coprod_{x \in k} k.$$ Then we can define a
structure of a ring scheme on $\underline{\mathbb{F}_q} \times \alpha_p$ by
defining the multiplication to be $(x,y)(z,w) = (xz,xw+yz)$ for
$(x,y),(z,w) \in (\underline{\mathbb{F}_q} \times \alpha_p) (R).$ This is a
non-reduced ring scheme.
\end{example}
\noindent  A famous example is also that the functor taking a ring $R$ to $W(R),$
its ring of big Witt vectors, is also representable by a ring scheme.\\
\noindent Let us note that we can form the category of $k-k$-birings, with morphisms between objects those morphism of $k$-algebras respecting the biring structure. We let $BR_{k,k}$ be the category of $k-k$-birings.
Let us recall the following definition from \cite{BW}.
\begin{definition}
Let $A$ be a $k-k$-biring. Then the functor $$\Ringk(A,-):\Alg \rightarrow
\Alg$$ has a left adjoint, $$A \odot_k -:\Alg \rightarrow \Alg.$$
Explicitly, for a $k$-algebra $R,$ $A \odot R$ is the $k$-algebra
generated by all symbols $a \odot r$ subject to the conditions that: \\
\begin{itemize}

\item $\forall a,a' \in A,$ $r \in R,$ $aa' \odot r = (a \odot r )(a'
\odot r).$

\item $\forall a,a' \in A,$ $r \in R,$ $(a+a') \odot r = (a \odot r) +
(a' \odot r).$

\item $\forall c \in k,$ $c \odot r = c.$

\item $\forall a \in A,$ $r,r' \in R,$ $a \odot (r+r') = \sum_i
(a_i^{(1)} \odot r)(a_i^{(2)} \odot r').$

\item $\forall a \in A,$ $r,r' \in R,$ $a \odot rr' = \sum_i (a_i^{[1]}
\odot r)(a_i^{[2]} \odot r').$

\item $\forall a \in A,$ $c \in k,$ $a \odot c = \beta(c)(a).$

\end{itemize}

\end{definition}

\noindent It is easy to see that $(\otimes_i A_i) \odot R \cong \otimes_i (A_i
\odot R)$ and that $A \odot (\otimes_i R_i ) \cong \otimes_i (A \odot
R_i).$ If further, $R$ is a $k-k$-bialgebra, we note that $A \odot R$ is
a $k-k$-bialgebra. Indeed, we have that $\Ringk(A \odot R,S) \cong
\Ringk(R,\Ringk(A,S))$ and since the latter set has a ring structure, so
does the former. One then verifies that $\odot_k$ gives a monoidal
structure to $BR_{k,k}.$ The unit of this monoidal structure is $k[e].$
$BR_{k,k}$ is a monoidal category, but it is not symmetric. Now, the Yoneda embedding sets up an equivalence of categories between the category of representable endofunctors $\Algk \rightarrow 
\Algk$ and $BR_{k,k}$ and under this equivalence, $\odot$ corresponds to $\circ,$ composition of representable endofunctors as given in the introduction. Denote the category of representable endofunctors $\Alg_k \rightarrow \Algk$ by $\Algkend.$
\begin{definition}
A $k$-plethory is a comonoid in $\Algkend$ where the monoidal structure is composition of endofunctors. Explicitly, on the level of representing objects, a $k$-plethory $P$ is a monoid in $BR_{k,k}.$ This means that $P$ is a
biring together with an associative map of birings $P \odot P
\rightarrow P$ and a unit $k[e] \rightarrow P.$

\end{definition}

\begin{remark} \label{prings}

For a plethory $P$ one can define an action of $P$ on a $k$-ring $R$ to
be a map $\circ: P \odot R \rightarrow R$ such that $(p_1 \odot p_2)
\circ r = p_1 \odot (p_2 \circ r)$ and $e \circ r = r, \forall p_1,p_2
\in P, r\in R.$ A ring $R$ together with an action of $P$ on $R$ is called a $P$-ring.

\end{remark}

\begin{example}

If $k$ is a finite ring, then $k^k,$ the set of functions $k \rightarrow
k$ is a plethory where $\circ$ is given by composition of functions.

\end{example}
\section{Classification of plethories over a field of characteristic zero.}

\noindent In this section we will prove that all plethories over a field of characteristic zero are linear. This question was asked by Bergman-Hausknecht \cite{BergmanHaus} and Borger-Wieland \cite{BW}.  To understand what it means for a plethory to be linear, we will introduce some terminology.
\begin{definition}
Let $A$ be a cocommutative bialgebra (not necesarily commutative) over $k$ with comultiplication $\Delta.$  Then there is a free $k$-plethory on $A$ over $k.$ The underlying algebra structure is $\Sym(A),$ the symmetric algebra on $A,$ and the coaddition $$\Delta^+: \Sym(A) \rightarrow \Sym(A) \otimes \Sym(A)$$ is induced from the map $$A \rightarrow \Sym(A) \otimes \Sym(A)$$ sending $a$ to $a \otimes 1 + 1 \otimes a.$ The comultiplication $\Delta^\times$ is similarily induced from $\Delta.$ The plethysm $$\circ: \Sym(A) \odot \Sym(A) \rightarrow \Sym(A)$$ is given by $$\Sym(A) \odot \Sym(A) \cong \Sym(A \otimes A) \xrightarrow{\Sym(m)} \Sym(A)$$ where $m$ is the multiplication on $A.$ Among the pairs consisting of a plethory $P$ and a morphism of bialgebras $f:A \rightarrow P$ the pair $\Sym(A)$ and $j:A \rightarrow \Sym(A)$ is initial with this property.
\end{definition}
\noindent Call a plethory $P$ linear if  $P \cong \Sym(A)$ for some bialgebra $A.$ The reason for calling it linear is that if $P \cong \Sym(A)$ for some bialgebra $A$ then $$\Ringk(\Sym(A),-) = \Modk(A,-).$$
Let us note now that by Theorem \ref{ringunipotent} any connected ring scheme of finite type is unipotent. Over $\mathbb{Q}$ (or more generally any field of characteristic zero) all group schemes are reduced by a theorem of Cartier. We say that a group scheme $G$ is étale if $G$ is a finite scheme and  geometrically reduced. This is equivalent to asking for the underlying Hopf algebra $A_G$ to be an étale algebra. Let us recall the following definition from the theory of group schemes (see for example \cite[II, $\mathsection 5,$ Proposition 1.8]{DemazureG} or \cite[Definition 9.4]{MilneiAG})
\begin{definition}
Let $G$ be a group scheme of finite type over $k.$ Let $A_G$ be the underlying Hopf algebra of $G$ and consider the largest étale k-subalgebra $\pi_0(A_G)$ of $A_G.$ $\pi_0(A_G)$ then has a Hopf algebra structure induced from the one on $A_G$ and we let $\pi_0(G) = \Spec \pi_0(A_G)$ be the group scheme associated to this Hopf algebra. 
\end{definition}
\noindent Note that there is a canonical map $G \rightarrow \pi_0(G).$ It is easy to see that if $\pi_0(G)= \Spec k,$ then $G$ is geometrically connected since in that case $A_G$ has no nontrivial idempotents.
\begin{lemma}
Any $k$-algebra scheme $\mathcal{R}$ of finite type over any infinite field $k$ is geometrically connected. 
\end{lemma}
\begin{proof}
Consider the connected-étale exact sequence $$0 \rightarrow \mathcal{R}^\circ \rightarrow \mathcal{R} \rightarrow \pi_0(\mathcal{R}) \rightarrow 0$$ of group schemes where $\mathcal{R}^\circ$ is the identity component of $\mathcal{R}.$ We will first show that $\pi_0(\mathcal{R})$ has a natural $k$-algebra scheme structure. Indeed, for this it is enough to show that $\mathcal{R}^\circ$ is a $k$-ideal scheme in $\mathcal{R}$. Let us start by proving that $m(\mathcal{R}^\circ \times \mathcal{R}) \subset \mathcal{R}^\circ.$ We know that the multiplication $$m:\mathcal{R}^\circ \times \mathcal{R} \rightarrow \mathcal{R}$$  takes the additive identity $e \in \mathcal{R}(k)$ to itself, i.e $m(e,x)=e$ for any $x \in \mathcal{R}(k).$ Further, the $k$-algebra structure on $\mathcal{R}^\circ$ is induced from the $k$-algebra structure on $\mathcal{R}.$ This clearly implies that $\mathcal{R}^\circ$ is a $k$-ideal scheme. Thus, the quotient $\mathcal{R}/\mathcal{R}^\circ \cong \pi_0(\mathcal{R})$ is a $k$-algebra scheme. Let us see that $\pi_0(\mathcal{R})$ is isomorphic to $\Spec k.$ One knows that the underlying algebra of $\pi_0(\mathcal{R})$ is a product of finite separable $k$-extensions. We consider $\Schk(\pi_0(\mathcal{R}),\pi_0(\mathcal{R})),$ this is a $k$-algebra (since $\pi_0(\mathcal{R})$ is a ring scheme). Because the underlying algebra of $\pi_0(\mathcal{R})$ is a finite product of finite separable field extensions, $\Schk(\pi_0(\mathcal{R}),\pi_0(\mathcal{R}))$ is a finite set.   However, for a finite set to have a $k$-algebra structure it must just contain one element, i.e it has to be the zero ring. This implies that $\pi_0(\mathcal{R}) = \Spec k$ so  $\mathcal{R}$ is geometrically connected. 
\end{proof}
\noindent Now, let us consider a Hopf algebra $H$ and denote the primitive elements of $H$ by $\Prim(H).$ We say that a Hopf algebra is primitively generated if $\Prim(H)$ generates $H$ as an algebra. Over characteristic zero all unipotent affine group schemes are primitively generated. We then have the classical Milnor-Moore theorem \cite{Milnor65}.
\begin{theorem}
For any commutative connected affine unipotent group scheme $H$ over a field of characteristic zero, the canonical map $$\Spec H \rightarrow  \Spec \Sym(\Prim(H))$$ is an isomorphism of group schemes. In particular, the underlying scheme is affine space.
\end{theorem}
\begin{remark}
Let us note that we can view $\Prim(H)$ as a Lie algebra with trivial commutator. Then the construction $\Sym(\Prim(H))$ is the same as the universal enveloping Lie algebra of $\Prim(H).$ 
\end{remark}
\noindent In \cite{BW} it is shown that if $Q$ is a plethory over a field $k$, then $\Prim(Q).$ the primitives with respect to $\delta^+_Q,$ is a cocommutative $k$-bialgebra. Briefly, the multiplication in $\Prim(Q)$ is given by the plethysm $\circ$ and the maps $$\Delta^\times:Q \rightarrow Q \otimes Q,$$  $\epsilon^\times:Q \rightarrow k$ induces a comultiplication respectively a counit on $\Prim(Q)$ making it a cocommutative counital bialgebra.   
\begin{theorem}
Let $Q$ be a plethory over a field of characteristic zero $k$. Then $Q$ is linear, i.e $$Q \cong \Sym(\Prim(Q))$$ where $\Sym(\Prim(Q))$ has the plethory structure as given in Definition 3.1. 
\end{theorem}

\begin{proof}
Suppose that $Q$  is a plethory over $k$. $\Prim(Q)$  naturally has a bialgebra structure as explained above. Given this, we can form the free plethory on $\Prim(Q),$ $\Sym(\Prim(Q)).$ We always have a natural map $$v:\Sym(\Prim(Q)) \rightarrow Q$$ of Hopf algebras, and this is bijective by Milnor-Moore. Thus to show that any plethory is linear, it suffices to show that this is actually a morphism of plethories. But this is clear: the pair $\Sym(\Prim(Q))$ and $$j:\Prim(Q) \rightarrow \Sym(\Prim(Q))$$ is initial in the category of pairs consisting of a plethory $P$ and a morphism $f:\Prim(Q) \rightarrow P$ of bialgebras. It is immediate that the canonical map  $v$ is induced by this universal property, when we note that there clearly is a map $\Prim(Q) \rightarrow Q$ of bialgebras. We will of course need to show that $v$ is an isomorphism in the category of plethories. This follows easily from the fact that $v$ is an isomorphism of affine schemes and thus has an inverse in the category of affine schemes. What remains to be checked is that this inverse is a morphism of plethories, but this is immediate since $v$ is.
\end{proof}
\begin{corollary}
Let $k$ be a field of characteristic zero. Then the category of plethories over $k$ is equivalent to the category of cocommutative $k$-bialgebras.
\end{corollary}
\begin{proof}
By the above theorem, the counit map is an isomorphism and it is immediate to see that the unit map is an isomorphism as well. 
\end{proof}
\begin{remark}
If $k$ is a field of characteristic zero and $Q$ a plethory over $k,$ this shows that the category of $Q$-rings is equivalent to the category of rings with an action of the bialgebra $P(Q).$ 
\end{remark}
\section{Some classification results in characteristic $p >0.$}
\noindent In this section we will start a classification for plethories over a perfect field $k$ of characteristic $p.$ Our classification results here only apply to a certain class of plethories. We state future research directions, as well as give some ''pathological'' examples which a complete classification must take into account. For any scheme $X$ over $k$  with structure map $f:X \rightarrow \Spec k$ we let $G^p $ be the pullback of $f$ along $F:\Spec k \rightarrow \Spec k,$ the Frobenius.
\\ \noindent \\ Let us briefly recall that for perfect fields $k,$ group schemes over $k$ have two especially important maps, the (relative) Frobenius $$F_G: G \rightarrow G^p \cong G$$ and the Verschiebung $$V_G:G \cong G^p \rightarrow G.$$ These satisfy the property that $F_GV_G=V_GF_G=p.$ A ring scheme $\mathcal{R}$ is called elementary unipotent if $V_{\mathcal{R}}=0,$ i.e the Verschiebung is zero. Call a plethory $Q$ weakly linear if there is a map of plethories $f:P \rightarrow Q$ where $P$ is a linear plethory (as defined in the previous section) such that $f$ when viewed as a map of algebras is surjective. This will, in particular, imply that $Q$ is primitively generated and is a quotient of $P$ by a $P-P$-ideal as defined in \cite{BW}. Not all plethories over a perfect field $k$ are primitively generated, as the following example shows (built on an example from \cite{Takeuchi}, Remark 1.6.2).
\begin{example}
Let $G$ be the group scheme $$\mathbb{G}_a \times_f \alpha_p$$ which as a scheme, is just $\mathbb{G}_a \times \alpha_p.$  We let the the group structure be given by $$(g_1,h_1)(g_2,h_2) = (g_1g_2,h_1+h_2+f(g_1,g_2))$$  for $$g_1,g_2,h_1,h_2 \in \mathbb{G}_a(R) \times \alpha_p(R)$$ where $f(x,y) = ((x+y)^p-x^p-y^p)/p.$ This is a $p$-torsion group scheme but is not elementary unipotent. One can define a non-unital ring scheme structure on $G$ be definining the multiplication to be trivial and then, when $k$ is finte, i.e $k \cong \mathbb{F}_q$  ''unitalize'' this by taking the direct product with $$\underline{\mathbb{F}_q} = \coprod_{a \in \mathbb{F}_q} \mathbb{F}_q$$ to get a ring scheme, as we did in Example 3.3. The underlying group scheme of this ring scheme is clearly not elementary unipotent, since the Verschiebung acts on each factor separately. Taking the free plethory on a biring (see \cite{BW} 2.1 ) will then give us a plethory with its underlying group scheme not elementary unipotent.
\end{example}
\noindent Another feature which differs from the case over a field of characteristic zero is that there are plethories which have a non-trivial multiplicative subgroups. This stems from the fact that there are ring schemes with non-trivial multiplicative subgroups.
\begin{example}
Consider $\mu_p = \Spec k[x,x^{-1}]/(x-1)^p$ with comultiplication $\Delta:x \rightarrow x \otimes x$ and counit $\epsilon(x) = 1.$ This is an example of a multiplicative group scheme which is p-torsion and we can as before define a trivial multiplication on $\mu_p,$ making it a non-unital ring scheme. We can then as previously stated, for finite fields, unitalize it to get a ring scheme by taking the direct product with $$\underline{\mathbb{F}_q}$$and after that we can form the free plethory to get a plethory $Q$ with a non-trivial multiplicative subgroup.The fact that it has a non-trivial multiplicative subgroup comes from , for example, the fact that there is a non-zero homomorphism of group schemes $\mu_p \rightarrow Q.$  
\end{example}
\noindent These two examples are rather artificial, but they show that plethories behave wildly different in characteristic $p >0$ than in characteristic $0.$
We know that for any group scheme $G$ over a perfect field $k$ of characteristic $p>0 ,$ the group $\Prim(G)$ of primitive elements has a natural action of the Frobenius, taking $x \in \Prim(G)$ to $x^p.$ In fact, $\Prim(G)$ becomes a module over a certain ring. As we previously stated, $\Prim(G) = \Hom(G, \mathbb{G}_a).$ We thus have that $\Prim(G)$ is naturally a module over the endomorphism ring $\End(\mathbb{G}_a,\mathbb{G}_a).$ 
\begin{definition}
Let $k\langle F \rangle$ be the non-commutative polynomial ring over $k$ in one variable $F$ with multplication given by, for $a \in k$ $aF = F_k(a)a$ where $F_k$ is the Frobenius endomorphism of $k.$
\end{definition}
\noindent It is a quick calculation to show that $\End(\mathbb{G}_a,\mathbb{G}_a) \cong k\langle F \rangle.$ We now see that $\Prim(G)$ is a module over $k\langle F \rangle.$  Let us denote the category of modules over $k\langle F \rangle$ by $Mod_{k\langle F \rangle}.$ Given a $k\langle F \rangle$-module $M$ one can construct an elementary unipotent group scheme $\Sym^{[p]}(M)$ as follows (for details we refer the reader to \cite{MilneiAG}) . Form $\Sym(M),$ the symmetric algebra on $M,$ with its obvious Hopf algebra structure and consider the map $j:M \rightarrow \Sym(M).$ We then quotient out by the ideal generated by the elements $$j(Fx)-j(x)^p$$ to get $\Sym^{[p]}(M).$  One notes that for any commutative algebraic group $G$ one always has a map $G \rightarrow \Sym^{[p]}(\Prim(G)).$ We have the following classical theorem (see \cite[IV,$\mathsection 3,$ Proposition $6.6$]{DemazureG})
\begin{theorem}
Let $G$ be an affine group scheme. The following are equivalent: \\
(i) The Verschiebung $V_G$ is zero. \\
(ii) $G$ is a closed subgroup of $\mathbb{G}_a^r$ for some $r.$ \\
(iii) The canonical homomorphism $G \rightarrow \Sym^{[p]}(\Prim(G))$ is an isomorphism.
\end{theorem}
\begin{remark}
What we call $\Sym^{[p]}(\Prim(Q))$ is the same as the enveloping $p$-algebra (also called the restricted universal enveloping algebra) on the $p$-Lie algebra $\Prim(Q)$ where $\Prim(Q)$ has trivial commutator.
\end{remark}
\begin{lemma}
When $Q$ is a plethory, then $\Sym^{[p]}(\Prim(Q))$ has the structure of a plethory.
\end{lemma}
\begin{proof}
We know that $\Prim(Q)$ has a $k\langle F \rangle$ module structure where the action of $F$ is just taking the $p$th power. Further, $\Sym^{[p]}(\Prim(Q))$ is the quotient of $\Sym(\Prim(Q)),$  which we know is a plethory, by the ideal $J$ generated by $j(x)^{p}-j(x^p)$, where $j:\Prim(Q) \rightarrow \Sym(\Prim(Q))$ is the inclusion in degree $1.$ It now suffices to show that this is a $Q-Q$-ideal (see \cite{BW} 6.1) for $U^{[p]}(\Prim(Q))$ to be a plethory. This is equivalent to showing that for a generating set $S$ of $J$ that $$\Delta^+_Q(S) \subset Q \otimes J + J \otimes Q,$$  $$\Delta_Q^{\times}(X) \subset Q \otimes J + J \otimes Q,$$ and $$\beta_Q(c)(S)=0$$  $\forall c \in k$ and that $$\Prim(Q) \odot X \odot Q \subset J.$$ The first is immediate, since taking $S$ to be the set of all $j(x)^{p}-j(x^p),$ we have $$\Delta^+(j(x)^{p})-\Delta^+(j(x^p)) = \Delta^+(j(x))^{p}-(j(x^p) \otimes 1 + 1 \otimes j(x^p))$$ which is equal to $$j(x)^{p} \otimes 1 + 1 \otimes j(x)^{p}  -j(x^p) \otimes 1 - 1 \otimes j(x^p) \subset J \otimes Q + Q \otimes J.$$ Further, $$\Delta^\times(j(x)^{p})-\Delta^\times(j(x^p)) = \sum_i j(x_i^{[1]})^{p} \otimes j(x_i^{[2]})^{p} - \sum_i j((x_i^{[1]})^p) \otimes j((x_i^{[2]})^p)$$  and this is equal to $$\sum_ij(x_i^{[1]})^{p} \otimes (j(x_i^{[2]})^{p}-j((x_i^{[2]})^p)) + \sum ((j((x_i^{[1]})^p)-j(x_i^{[1]})^p) \otimes j((x_i^{[2]})^p))$$ but this is in $J \otimes P + P \otimes J.$  We also need to show that $\beta_Q(c)(S)=0,$ this is clear.  The last containment is similarily easy to verify.
\end{proof}
\begin{theorem}
When $Q$ is a plethory over a perfect field such that $V_Q=0,$ then $\Sym^{[p]}(\Prim(Q)) \cong Q.$  We then say that $Q$ is a $p$-linear plethory.
\end{theorem}
\begin{proof}
All one has to verify is that the canonical map $f:Q \rightarrow \Sym^{[p]}(\Prim(Q))$ is a map of plethories. But this is obvious since this map is just the composition of the two plethory maps $Q \rightarrow \Sym(\Prim(Q))$ and $\Sym(\Prim(Q)) \rightarrow \Sym^{[p]}(\Prim(Q)).$  
\end{proof}
\begin{remark}
We have seen that plethories need not be elementary unipotent and not purely unipotent either (i.e it can have a non-trivial multiplicative subgroup) Let us note that there can be no non-trivial finite plethories over an infinite perfect field $k.$ Indeed, from what we have seen all plethories $Q$ are connected over an infinite field. By classical Dieudonné theory we can then decompose $Q$ as $Q = Q^{loc,red} \times Q^{loc,loc}.$ This would imply that the Frobenius is nilpotent, but this can never happen: the Frobenius is always a map of ring schemes. 
\end{remark}
\noindent It seems to us that to classify plethories over a perfect field one should establish an extension of ordinary Dieudonné theory to account for ring schemes, which has been done to some extent by Hedayatzadeh in \cite{HadiT} and for Hopf rings by Goerss \cite{GoerssH} and Buchstaber-Lazarev \cite{Buchstaber}. Note that Hedayatzadeh work with finite / profinite group schemes and with local group schemes, which  limits their applications to ring schemes since we have seen that there are no non-trivial finite connected ring schemes over a perfect field. 

\bibliography{ref}{}
\bibliographystyle{plain}

\end{document}